\documentclass[letterpaper,conference]{ieeeconf}  

\IEEEoverridecommandlockouts                              
\usepackage{graphicx} 
\usepackage{mathptmx} 
\usepackage{times} 
\usepackage{amsmath} 
\usepackage{amssymb}  

\usepackage[noadjust]{cite}

\newtheorem{definition}{Definition}[section]
\newtheorem{assumption}[definition]{Assumption}
\newtheorem{lemma}[definition]{Lemma}
\newtheorem{proposition}[definition]{Proposition}
\newtheorem{theorem}[definition]{Theorem}
\newtheorem{remark}[definition]{Remark}
\newtheorem{corollary}[definition]{Corollary}
\newtheorem{example}[definition]{Example}

\newcommand{\norm}[1]{\lVert #1 \rVert}

\newcommand{\nnorm}[1]{{\left\vert\kern-0.25ex\left\vert\kern-0.25ex\left\vert #1 \right\vert\kern-0.25ex\right\vert\kern-0.25ex\right\vert}}

\DeclareMathOperator{\diag}{diag} 
\DeclareMathOperator{\myvec}{vec} 
\renewcommand{\vec}{\myvec}

\DeclareMathOperator\conv{conv}
\DeclareMathOperator\supp{supp}
\DeclareMathOperator\interior{int}

\newenvironment{proofof}[1]{
\begin{proof}}{\end{proof}}

\newenvironment{proofsk}{
\begin{proof}}{\end{proof}}

\newcommand{\iref}[1]{\ref{#1})}

\newcommand{\bm}[1]{\begin{bmatrix}#1\end{bmatrix}}

\newcommand{\myscale}{.17}

\newcommand\hyperrefopt{bookmarks=true,bookmarksnumbered=true,
pdfpagemode={UseOutlines},plainpages=false,pdfpagelabels=true,
colorlinks=true,linkcolor={black},citecolor={black},urlcolor={black},
pdftitle={}, pdfsubject={}, pdfauthor={Masaki Ogura}, pdfkeywords={}}
\usepackage[\hyperrefopt]{hyperref}

\title{\LARGE \bf
On the Mean Stability of a Class of Switched Linear Systems
}

\author{Masaki Ogura and Clyde F.~Martin
\thanks{M.~Ogura and C.~F.~Martin are with the Department of Mathematics and Statistics, Texas Tech University, Broadway and Boston, Lubbock, TX 79409-1042, USA.  
        {\tt\small Masaki.Ogura@ttu.edu}, {\tt\small Clyde.F.Martin@ttu.edu}}%
}

\newcommand{\afterequation}{\vskip 5pt}

\begin{document}

\maketitle
\thispagestyle{empty}
\pagestyle{empty}

\begin{abstract}
This paper investigates the mean stability of a class of  discrete-time
stochastic switched linear systems using the $L^p$-norm joint spectral
radius of the probability distributions governing the switched systems.
First we prove a converse Lyapunov theorem that shows the equivalence
between the mean stability and the existence of a homogeneous Lyapunov
function. Then we show that, when $p$ goes to $\infty$, the stability of
the $p$th mean becomes equivalent to the absolute asymptotic stability
of an associated deterministic switched system. Finally we study the
mean stability of Markovian switched systems. Numerical examples are
presented to illustrate the results. 
\end{abstract}


\section{Introduction}

This paper studies the discrete-time stochastic switched linear system of the form
\begin{equation} \label{eq:OurSystem}
\Sigma: x(k+1) = A_{k} x(k)
\end{equation}
where $x(k)$ represents a finite-dimensional state vector and
$\{A_k\}_{k=0}^\infty$ is a stochastic process taking values in the set
of square matrices of an appropriate dimension. One of its most natural
stability notions is almost sure stability~\cite{Kozin1969}, which
requires that the state $x(k)$ converges to the origin with probability
one as $k\to\infty$. However this stability is difficult to check in
practice because it is characterized by a quantity called the top
Lyapunov exponent, whose computation is in general a hard
problem~\cite{Tsitsiklis1997a}. For example, the necessary and
sufficient condition in \cite{Dai2008a} is one of the most tractable
conditions but cannot necessarily be checked with finite computation.

This difficulty has motivated many authors to study another stability,
called $p$th {\it mean stability}, which requires that the expected
value of the $p$th power of the norm of the state~$x(k)$ converges
to~$0$. It is well
known~\cite{Jungers2010,Fang2002a,FANG1994,Fragoso2005,Vargas2010} that
the $p$th mean stability can be often checked by computing the spectral
radius of a matrix that is easy to compute. Recently the authors
showed~\cite{Ogura2012b} that, if each $A_k$ follows an identical
probability distribution independently and either $p$ is even or the
distribution possess a certain invariance property, then the $p$th mean
stability of  the system \eqref{eq:OurSystem} is characterized by the
spectral radius of a matrix, generalizing the results obtained in the
literature \cite{Jungers2010,Fang2002a,FANG1994}.  Also it is shown that
the mean square stability is equivalent to the existence of a certain
quadratic Lyapunov function. The derivation of these results depends on
an extended version \cite{Ogura2012b} of {\it $L^p$-norm joint spectral
radius}~\cite{Protasov1997}.

Generalizing the above result on Lyapunov functions, in this paper we
first show that, for a general even exponent~$p$, the $p$th mean
stability of $\Sigma$ is equivalent to the existence of a homogeneous
Lyapunov function of degree~$p$. This equivalence is still true for a
general $p$ provided the distribution possesses a certain invariance
property. Imposing that the value of the function decreases along
trajectories only in {\it expectation} enables us to construct a
Lyapunov function even when the system is not stable for an arbitrary
switching signal, as opposed to the results in the
literature~\cite{Molchanov1989,Dayawansa1999,Doan2013},  Moreover our
Lyapunov function can be constructed easily by solving a linear matrix
inequality or an eigenvalue problem. 

Then we study a limiting behavior of the $p$th mean stability. 
It is well known \cite{Fang2002} that, roughly speaking, the
$p$th mean stability becomes equivalent to almost sure stability in  the
limit of $p \to 0$. As a counterpart of this fact we show that, in the
limit of $p\to\infty$, the $p$th mean stability is equivalent to the
stability of an associated {\it deterministic}
switched system for an arbitrary switching. This result will be proved by showing an extension of
the formulas \cite{Blondel2005,Xu2011} that express joint spectral
radius as the limit of $L^p$-norm joint spectral radius. 

Finally we extend the results in \cite{Ogura2012b} to Markovian switched
systems, where $A_k$ in \eqref{eq:OurSystem} is not  necessarily
identically and independently distributed.  Again assuming the
invariantce property of the Markov process we will give a
characterization of the $p$th mean stability. 

This paper is organized as follows.  After preparing necessary
mathematical notation and conventions, in
Section~\ref{sec:GJSR&Stability} we review the basic facts of the
stability of discrete-time linear switched systems.
Section~\ref{sec:LyapunovTheorem} proves a converse Lyapunov theorem.
Then in Section~\ref{sec:Limit} we study the limiting behavior of the
$p$th mean stability. Finally Section~\ref{sec:Markov} studies the mean
stability of Markovian switched systems.

\subsection{Mathematical Preliminaries}

Let $\mathbb{R}_+$ denote the set of nonnegative numbers.  The spectral
radius of a square matrix is denoted by $\rho(\cdot)$. A
subset~$K\subset \mathbb{R}^d$ is called a cone if $K$ is closed under
multiplication by nonnegative scalars. The cone is said to be solid if
it possesses a nonempty interior. We say that a cone is pointed if it
contains no line; i.e., if $x, -x \in K$ then $x=0$. We say that $K$ is
{\it proper} if it is closed, convex, solid, and pointed. For example
the positive orthant $\mathbb{R}^{d}_+$ of $\mathbb{R}^d$ is a proper
cone.  The dual cone $K^*$ is defined by
\begin{equation*}
K^* = \{ f\in\mathbb{R}^d : f^\top x \geq 0 \text{ for every $x\in K$}  \}. 
\end{equation*}
A matrix $M\in \mathbb{R}^{d\times d}$ is said to leave $K$ invariant,
written $M\geq^K 0$, if $MK\subset K$. A subset $\mathcal M \subset
\mathbb{R}^{d\times d}$ is said to leave $K$ invariant if any matrix in
$\mathcal M$ leaves $K$ invariant. For $M, N \in \mathbb{R}^{d\times
d}$, by $M\geq^K N$ we mean $M-N \geq^K 0$.  $M$ is said to be
$K$-positive, written $M>^K 0$, if $M(K - \{0\})$ is contained in the
interior $\interior K$ of $K$.  The next lemma collects some elementary
facts of cones and their duals. 
\begin{lemma}\label{lem:basic:M>^K0}
Let $K$ be a proper cone and let $M >^K 0$. 
\begin{enumerate}
\item 
$M$ has a simple eigenvalue $\rho(M)$, which is greater than the
magnitude of any other eigenvalue of $M$. Moreover the eigenvector 
corresponding to the eigenvalue $\rho(M)$ is in $\interior K$ (see,
e.g., \cite{Vandergraft1968});

\item 
$K^*$ is a proper cone and $M^\top$ is $K^*$-positive 
\cite[2.23]{Berman1979}. 
\end{enumerate}
\end{lemma}

A norm $\norm{\cdot}$ on $\mathbb{R}^d$ is said to be cone absolute
\cite{Seidman2005} with respect to a proper cone~$K$ if, for every
$x\in\mathbb{R}^d$, 
\begin{equation}\label{eq:def:coneabsolute}
\norm{x} = \inf_{\substack{v, w\in K\\x=v-w}}\norm{v+w}. 
\end{equation}
Also we say that $\norm{\cdot}$ is cone linear with respect to $K$ if
there exists $f\in K^*$ such that 
\begin{equation} \label{eq:def:conelinear}
\norm{x} = f^\top x \text{ for every }x\in K. 
\end{equation}
A norm that is cone absolute and cone linear with respect to a proper
cone is said to be cone linear absolute. Every $f\in \interior K^*$ yields
\cite{Seidman2005} the cone linear absolute norm determined by
\eqref{eq:def:conelinear} and \eqref{eq:def:coneabsolute}, and we denote
this norm by $\norm{\cdot}_f$. The next lemma lists some properties of
cone linear absolute norms. 
\begin{lemma}[\cite{Seidman2005}]\label{lemma:ProdMonotone}
Let $\norm{\cdot}$ be a cone linear absolute norm.  
\begin{enumerate}
\item The induced norm of $M \in \mathbb{R}^{d\times d}$, defined by
$\norm{M} = \sup_{x\in \mathbb{R}^d}\frac{\norm{Mx}}{\norm{x}}$, 
satisfies
\begin{equation}\label{eq:normMshort}
\norm{M} = \sup_{x\in K}\frac{\norm{Mx}}{\norm{x}}. 
\end{equation}

\item If $M \geq^K N \geq^K 0$ then $\norm{M} \geq \norm{N}$. 
 
\item If $M_i \geq^K N_i \geq^K 0$ for all $i=1, \dotsc, k$ then 
$\norm{M_1 \dotsi M_k} \geq \norm{N_1 \dotsi N_k}$. 
\end{enumerate}
\end{lemma}
\begin{proof}
The first two claims can be found in \cite{Seidman2005}. The last one
immediately follows from the second one. 
\end{proof}

We denote the Kronecker product (see, e.g., \cite{Brewer1978}) of
matrices $M$ and $N$ by $M\otimes N$.  For a positive integer $p$ define
the Kronecker power $M^{\otimes p}$ by $M^{\otimes 1} := M$ and
$M^{\otimes (p)} = M^{\otimes (p-1)}\otimes M$ recursively for a general
$p$. It holds~\cite{Brewer1978} that
\begin{equation}\label{eq:ABotimesp}
  (MN)^{\otimes p} = M^{\otimes p} N^{\otimes p}.
\end{equation}
Also for $\mathcal M \subset \mathbb{R}^{d\times d}$ define
${\mathcal M}^{\otimes p} := 
\{ M^{\otimes p}: M\in \mathcal M \} \subset \mathbb{R}^{d^p \times d^p}. $
The next lemma is proved in \cite{Blondel2005}. 
\begin{lemma}\label{lemma:KronProdLeave}
Let $\mathcal M \subset \mathbb{R}^{d\times d}$.  If $\mathcal M$ leaves
a proper cone $K$ invariant then ${\mathcal M}^{\otimes p}$  leaves the
proper cone 
\begin{equation}\label{eq:Kptilde}
\tilde K_{p} := \overline{\conv (K^{\otimes p})}
\end{equation}
invariant. 
\end{lemma}

Let $\mu$ be a probability distribution on $\mathbb{R}^{d\times d}$. The
support of $\mu$ is denoted by $\supp \mu$.  For a measurable function
$f$ on $\mathbb{R}^{d\times d}$ we denote the expected value of $f$ by
$E_\mu[f] = \int_{\mathbb{R}^{d\times d}} f(X)\,d\mu(X)$. The subscript
$\mu$ will be omitted when it is clear from the context. We define the
probability distribution $\mu^{\otimes p}$ on $\mathbb{R}^{d^p\times
d^p}$ as the image of the measure~$\mu$ under the
mapping~$(\cdot)^{\otimes p}\colon \mathbb{R}^{d\times d} \to
\mathbb{R}^{d^p\times d^p} \colon M\mapsto M^{\otimes p}$ (for detail
see, e.g., \cite{Bogachev2007}).

Finally we define the operator $\vec \colon \mathbb{R}^{m \times d} \to
\mathbb{R}^{md}$ by 
\begin{equation*}
\vec \begin{bmatrix}
M_1 & \cdots & M_d
\end{bmatrix}
:=
\begin{bmatrix}
M_1\\\vdots\\M_d
\end{bmatrix},\ M_1, \dotsc, M_d \in \mathbb{R}^m. 
\end{equation*}
We extend this definition to the product space $(\mathbb{R}^m)^N$ as
\begin{equation*}
\vec(M_1, \dotsc, M_N) := \vec\begin{bmatrix}
M_1 & \cdots & M_N
\end{bmatrix}. 
\end{equation*}

\section{Stability of Switched Linear Systems}
\label{sec:GJSR&Stability}

This section reviews the stability of switched linear systems, namely,
the relationship between mean stability and $L^p$-norm joint spectral
radius \cite{Ogura2012b}, and that between absolute asymptotic stability
and joint spectral radius.  Throughout this paper $p$ will denote a
positive integer and $\mu$ will denote a probability distribution
on~$\mathbb{R}^{d\times d}$ with compact support.  Unless otherwise
stated, $A$ and $A_k$ ($k=0, 1, \dotsc$)  will denote random variables
that follow $\mu$ independently. $\norm{\cdot}$ will denote a norm on
$\mathbb{R}^d$ or $\mathbb{R}^{d\times d}$. 

Consider the stochastic linear switched system
\begin{equation*}
\Sigma_\mu : x(k+1) = A_k x(k),\ \text{$A_k$ follows $\mu$ independently}. 
\end{equation*}
The solution of $\Sigma_\mu$ with the initial state $x(0) = x_0$ is
denoted by $x(\cdot;x_0)$. 

\begin{definition} \label{def:MeanStability}
We say that $\Sigma_\mu$ is 
\begin{enumerate}
\item 
{\it exponentially stable in $p$th mean} ({\it
$p$th mean stable} for short) if there exist  $M>0$ and
$\beta>0$ such that, for any initial state $x_0$,
\begin{equation*}
E[\norm{x(k;x_0)}^p] \leq Me^{-\beta k}\norm{x_0}^p;
\end{equation*}

\item  
{\it exponentially stable in $p$th moment} ({\it
$p$th moment stable} for short) if there exist $M, \beta
>0$ such that, for every $x_0$, 
\begin{equation*}
\norm{E[x(k;x_0)^{\otimes p}]}
\leq
M e^{-\beta k} \norm{x_0}^p.
\end{equation*}
\afterequation
\end{enumerate}
\end{definition}

The mean stability of $\Sigma_\mu$ is closely related \cite{Ogura2012b}
to the $L^p$-norm joint spectral radius of $\mu$ defined as follows. 
\begin{definition}[\cite{Ogura2012b}]
The {\it $L^p$-norm  joint spectral radius} ({\it
$p$\nobreakdash-radius} for short) of~$\mu$ is defined by
\begin{equation}\label{eq:def:pthradius}
\rho_{p, \mu}
:=
\lim_{k\to\infty} \left(E[\norm{A_k \dotsi A_1}^p ]\right)^{1/pk}. 
\end{equation}
\afterequation
\end{definition}

This definition extends \cite{Ogura2012b} the classical $L^p$-norm joint
spectral radius (see, e.g., \cite{Jungers2011b}) for a finite set of
matrices.  By the compactness of the support of $\mu$, the $p$-radius
$\rho_{p, \mu}$ is well-defined and is finite \cite{Ogura2012b}. Also,
by the equivalence of the norms on $\mathbb{R}^{d\times d}$, the value
of $p$-radius does not depend on the norm used in
\eqref{eq:def:pthradius}. 

The next proposition summarizes the characterization of the $p$th mean
stability obtained in \cite{Ogura2012b}. 
\begin{proposition}[{\cite{Ogura2012b}}] \label{prop:stabilitychar}
Assume that either 
\begin{enumerate}\renewcommand{\theenumi}{\alph{enumi}}
\item $p$ is even or 
\item $\supp \mu$ leaves a proper cone invariant. 
\end{enumerate}
Then the following conditions are equivalent:
\begin{enumerate}
\item $\rho_{p, \mu} < 1$;
\item $\Sigma_\mu$ is $p$th mean stable;
\item $\Sigma_\mu$ is $p$th moment stable. 
\end{enumerate}
Moreover it holds that
\begin{equation*}
\rho_{p, \mu} = \left(\rho(E[A^{\otimes p}])\right)^{1/p}. 
\end{equation*}
\afterequation
\end{proposition}

We will also need the next lemma that lists basic properties of
$p$-radius. 
\begin{lemma}[{\cite{Ogura2012b}}]\label{lemma:BasicProertiesOfPradius}
\ 
\begin{enumerate} 
\item  $\rho_{p, \mu}$ is non-decreasing with respect to $p$. 

\item  For all $p$ and $k$ it holds that 
\begin{equation} \label{eq:connects_p_and_p/k}
\rho_{p,\mu} = \left[ \rho_{p/k, \mu^{\otimes k}} \right]^{1/k}. 
\end{equation}
\afterequation
\end{enumerate}
\end{lemma}

Let us also review the notion of joint spectral
radius~\cite{Jungers2009}. Let $\mathcal M$ be a subset of
$\mathbb{R}^{d\times d}$. The joint spectral radius of  $\mathcal M$ is
defined by
\begin{equation*}
\hat \rho(\mathcal M) 
:=
\limsup_{k\to\infty} \sup_{A_1, \dotsc, A_k \in \mathcal M} 
\norm{A_k \dotsi A_1}^{1/k}. 
\end{equation*}
Again this quantity is independent of the norm $\norm{\cdot}$. The joint
spectral radius is known to characterize the stability of the {\it
deterministic} switched system
\begin{equation*}
\Sigma_{\mathcal M} : x(k+1) = A_k x(k),\ A_k \in \mathcal M. 
\end{equation*}
$\Sigma_{\mathcal M}$ is said to be {\it absolutely asymptotically
stable} if $x(k) \to 0$ as $k\to\infty$ for every possible switching
pattern.  The next proposition is well known (for its proof see, e.g.,
\cite{Theys2005}). 
\begin{proposition}\label{prop:AAS:characterization}
$\Sigma_{\mathcal M}$ is absolutely asymptotically stable if and only if
$\hat\rho(\mathcal M) < 1$.
\end{proposition}

\section{Converse Lyapunov Theorem} \label{sec:LyapunovTheorem}

The aim of this section is to show a converse Lyapunov theorem for the
switched system $\Sigma_\mu$.  Let us begin by defining  Lyapunov
functions for $\Sigma_\mu$. 
\begin{definition}
A continuous and positive definite function $V\colon \mathbb{R}^d \to
\mathbb{R}$ is said to be a {\it Lyapunov function} for $\Sigma_\mu$ if
there exists $0\leq \gamma<1$ such that 
\begin{equation} \label{eq:defn:LyapunovFunction}
E[V(Ax)] \leq \gamma V(x)
\end{equation}
for every $x\in\mathbb{R}^d$. 
\end{definition}

The next theorem is the main result of this section.
\begin{theorem}\label{thm:lyapu}
Assume that either
\begin{enumerate} \renewcommand{\theenumi}{\alph{enumi}}
\item $p$ is even or\label{item:peven}

\item $\supp \mu$ leaves a proper cone $K$ invariant and moreover $E[A^{\otimes
p}] >^{\tilde K_{p}} 0$. \label{item:invcone}
\end{enumerate}
Then $\Sigma_\mu$ is $p$th mean stable if and only if it admits a
homogeneous Lyapunov function of degree $p$. 
\end{theorem}

\begin{remark}
The assumptions in Theorem~\ref{thm:lyapu} are needed because  the
theorem relies on Proposition~\ref{prop:stabilitychar}. The relaxation
of those assumptions is left as an open problem. A possible approach can
be found in \cite{Jungers2011b}, where the authors proposes a method to
approximately compute $p$-radius without such assumptions. 
\end{remark}

It is straightforward to prove sufficiency. 

\begin{proofof}{sufficiency in Theorem~\ref{thm:lyapu}}
Assume that $\Sigma_\mu$ admits a homogeneous Lyapunov function of
degree $p$. Let $x_0 \in \mathbb{R}^d$ be arbitrary. Using induction we
can show $E[V(x(k;x_0))] \leq  \gamma^k V(x_0)$.  Since $V$ is
continuous, homogeneous with degree $p$, and positive definite, there
exist constants $C_1, C_2 > 0$ such that $C_1\norm{x}^p \leq V(x) \leq
C_2\norm{x}^p$ for every $x\in \mathbb{R}^d$. Therefore
$E[\norm{x(k;x_0)}^p] \leq (C_2/C_1) \gamma^k \norm{x_0}^p$.  Thus
$\Sigma_\mu$ is $p$th mean stable. 
\end{proofof}

The rest of this section is devoted for the proof of necessity. The
proof for the case \iref{item:peven} depends on its special case $p=2$,
which is proved in \cite{Ogura2012b} and is quoted as the next
proposition for ease of reference. 
\begin{proposition}[\cite{Ogura2012b}] \label{prop:squaremeanLyapunov}
$\Sigma_\mu$ is mean square stable if and only if it admits a quadratic
Lyapunov function of the form~$x^\top Hx$, where $H$ is a positive
definite matrix. Moreover such an $H$ can be obtained by solving a
linear matrix inequality.
\end{proposition}

To prove the second case \iref{item:invcone} we will need the next
proposition. 
\begin{proposition}\label{prop:rho(M)}
Let $K \subset \mathbb{R}^d$ be a proper cone and assume that $M >^K 0$.
Then there exists $f \in \interior{(K^*)}$ that induces the cone linear
absolute norm $\norm{\cdot}_f$ such that  $\norm{M}_f = \rho(M)$. 
\end{proposition}
\begin{proof}
By Lemma~\ref{lem:basic:M>^K0} the matrix~$M$ admits the Jordan
canonical form  $J = V^{-1} M V$ where $V\in \mathbb{R}^{d\times d}$ is
an invertible matrix whose columns are the generalized eigenvectors of
$M$ and $J \in \mathbb{R}^{d\times d}$ is of the form 
\begin{equation*}
J = \begin{bmatrix}
J_0 &  0 \\
0   & \rho(A)
\end{bmatrix}
\end{equation*}
for some upper diagonal matrix $J_0$. Define $f \in \mathbb{R}^d$ by 
\begin{equation*} 
V^{-1} = \begin{bmatrix}
*\\f^\top
\end{bmatrix}. 
\end{equation*}
We can easily see that $f$ is an eigenvector of $M^\top$ corresponding
to the eigenvalue $\rho(M)$. Since $K^*$ is proper,  Lemma
\ref{lem:basic:M>^K0} shows $f \in \interior(K^*)$. Thus $f$ gives a
cone linear absolute norm with respect to $K$. Since for every $x\in K$
we have 
\begin{equation*}
\norm{Mx}_f = f^\top Mx = \rho(M) f^\top x = \rho(M) \norm{x}_f, 
\end{equation*}
the equation \eqref{eq:normMshort} immediately shows $\norm{M}_f =
\rho(M)$.
\end{proof}

Now we can complete the proof of Theorem~\ref{thm:lyapu}. 

\begin{proofof}{necessity in Theorem~\ref{thm:lyapu}}
First consider the case \iref{item:peven}. Assume that $\Sigma_\mu$ is
$p$th mean stable for $p = 2q$ where $q$ is a positive integer.  Then
Proposition~\ref{prop:stabilitychar} gives $\rho_{p, \mu} < 1$ and 
therefore $\rho_{2,\mu^{\otimes q}} < 1$ by
\eqref{eq:connects_p_and_p/k}. Hence, by Proposition
\ref{prop:squaremeanLyapunov}, $\Sigma_{\mu^{\otimes q}}$ admits a
homogeneous Lyapunov function $V$ of degree~$2$. Now define $W \colon
\mathbb{R}^d \to \mathbb{R}$ by $W(x) := V(x^{\otimes q})$. Since $V$ is
a Lyapunov function and is homogeneous of degree 2, $W$ is continuous,
positive definite, and is homogeneous of degree $2q = p$. Moreover, by
\eqref{eq:ABotimesp}, if $B$ follows $\mu^{\otimes q}$ then
\begin{equation*} 
\begin{aligned}
E[W(Ax)]
&=
E[V(A^{\otimes q}x^{\otimes q})]
\\
&=
E[V(B x^{\otimes q})]
\\
&\leq
\gamma V(x^{\otimes q})
\\
&=
\gamma W(x)
\end{aligned}
\end{equation*}
for some $\gamma < 1$ since $V$ is a Lyapunov function for
$\Sigma_{\mu^{\otimes q}}$. This shows that $W$ is a Lyapunov function
for $\Sigma_\mu$. 

Then let us turn to the second case~\iref{item:invcone}. We only
consider the special case $p=1$. Notice that $\tilde K_1 = K$. Assume
that $\supp \mu$ leaves a proper cone $K$ invariant and $E[A] >^K 0$.
If $\Sigma_\mu$ is (first) mean stable then $\gamma :=\rho(E[A]) < 1$ by
Proposition~\ref{prop:stabilitychar}. Also, by
Proposition~\ref{prop:rho(M)}, there exists a cone linear absolute norm
$\norm{\cdot}$ on $\mathbb{R}^{d}$ such that $\norm{E[A]} = \gamma$. Now
we define  $V(x) := \norm{x}$. We need to show
\eqref{eq:defn:LyapunovFunction}. Let $x\in\mathbb{R}^d$ and
$\epsilon>0$ be arbitrary. Since $\norm{\cdot}$ is cone absolute there
exist $x_1, x_2\in K$ such that $x=x_1-x_2$ and $\norm{x_1+x_2} \leq 
\norm{x} +\epsilon$. Notice that, by the linearity of $\norm{\cdot}$ on
$K$, we have $E[\norm{Ax_i}] = \norm{E[A]x_i} \leq \gamma \norm{x_i}$.
Therefore 
\begin{equation*}
\begin{aligned}
E[V(Ax)]
&= 
E[\norm{Ax_1 - Ax_2}]
\\
&\leq
E[\norm{Ax_1}] + E[\norm{Ax_2}]
\\
&\leq
\gamma (\norm{x_1} + \norm{x_2})
\\
&
\leq
\gamma(\norm{x} +\epsilon)
\end{aligned}
\end{equation*}
again by the linearity of $\norm{\cdot}$ on $K$.  Since $\epsilon>0$ was
arbitrary the inequality \eqref{eq:defn:LyapunovFunction} actually holds
and this finishes the proof for $p=1$. The case for a general $p$ can be
proved in the same way as the first half of this proof with
Lemma~\ref{lemma:KronProdLeave}. 
\end{proofof}

\begin{example}\label{ex:lyapu}
Consider the probability distribution 
\begin{equation*}
\mu = \begin{bmatrix}
[0, 1.5]& [0,1.8]\\
[0,0.15]&[0,1.2]
\end{bmatrix}
\end{equation*}
where each closed interval indicate that the corresponding element
of~$\mu$ is the uniform distribution on the interval. Clearly $\supp \mu$
leaves the proper cone~$\mathbb{R}^{2}_+$ invariant. Also the mean
$E[A]$ has only positive entries so that it is
$\mathbb{R}^{2}_+$-positive. Therefore Theorem~\ref{thm:lyapu} shows
that $\Sigma_\mu$ admits a homogeneous Lyapunov function of degree $1$. 
From the proof of the theorem, the Lyapunov function is given as a cone
linear absolute norm $\norm{x}_f$ and we can obtain the $f$ by following
the proof of
Proposition~\ref{prop:rho(M)} as $f = \begin{bmatrix}0.3838&1\end{bmatrix}^\top$.

We generate $200$ sample paths of  $\Sigma_\mu$ with the initial
state~$x_0 = \bm{0&1}^\top$. Fig.~\ref{fig:Lyapunov} shows the sample
means of the Lyapunov function $\norm{x(k)}_f$ and the Euclidean norm
$\norm{x(k)}$. While the sample mean of our Lyapunov function is almost
decreasing, that of the Euclidean norm shows oscillation.
Fig.~\ref{fig:Level} shows the average of the sample paths and the
contour plot of our Lyapunov function and the Euclidean norm. 
\begin{figure}[tb]
\begin{center}
\includegraphics[scale=\myscale]{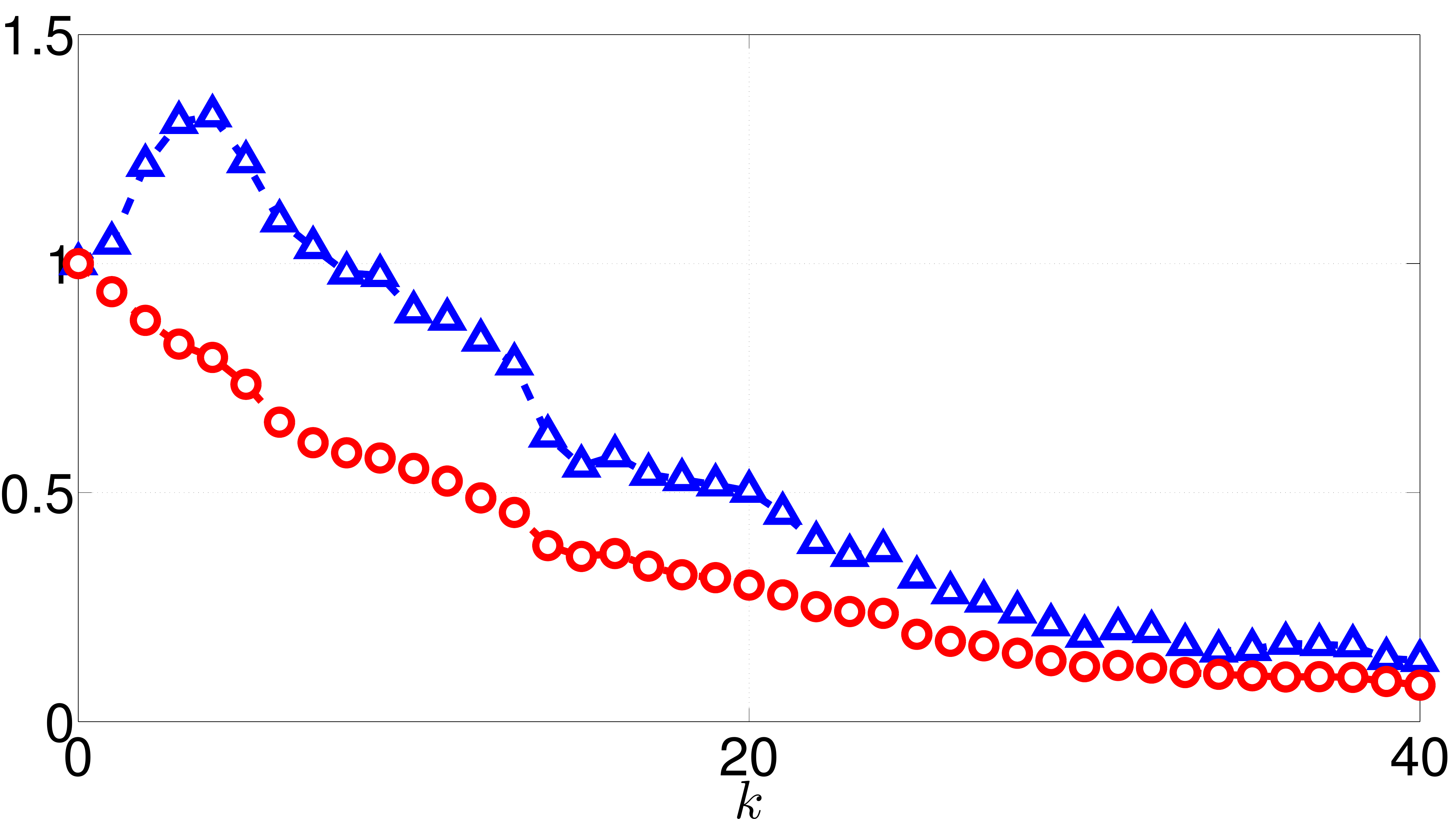}  \caption{The
sample means of the Lyapunov function (circle) and the Euclidean norm
(triangle)} \label{fig:Lyapunov}
\end{center}
\end{figure}

\begin{figure}[tb]
\begin{center}
\includegraphics[scale=\myscale]{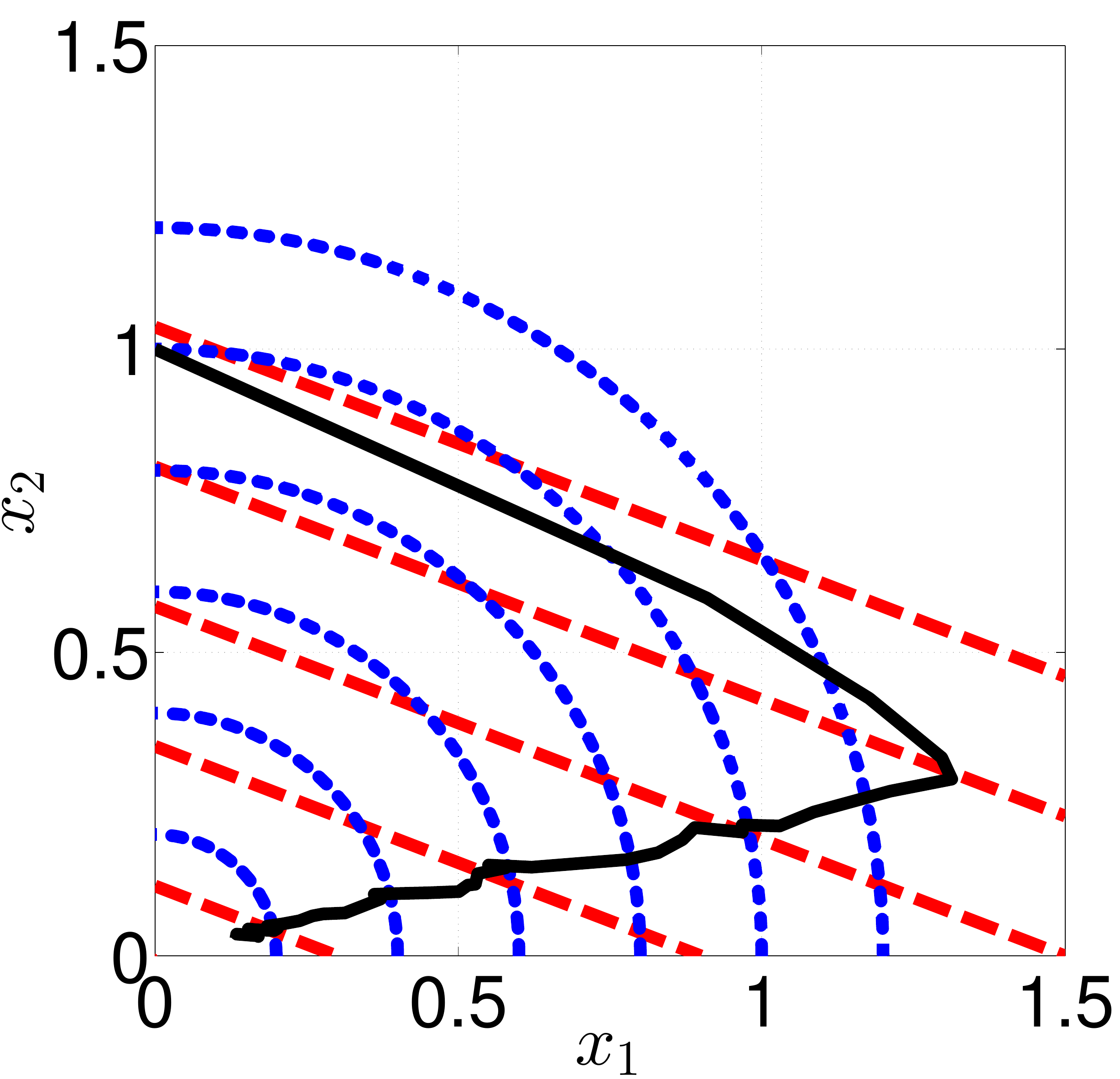}  \caption{The averaged
sample path (solid) and the level plots of the Lyapunov function
(dashed) and the Euclidean norm (dotted)} \label{fig:Level}
\end{center}
\end{figure}

\end{example}

\begin{remark}
Since $\supp \mu$ is infinite in Example~\ref{ex:lyapu}, we
cannot use the methods to construct a Lyapunov function for positive
systems proposed in the literature (see, e.g., \cite{Doan2013}).
\end{remark}

\section{Limiting Behavior of $p$th Mean Stability}\label{sec:Limit}

This section studies the limiting behavior of $p$th mean stability as
$p\to \infty$.  We start with the following observation. Let $\mu$ be a
probability distribution on $\mathbb{R}^{d\times d}$ and let $\mathcal M
= \supp \mu$.  Then the definitions of $p$-radius and joint spectral
radius show $\rho_{p,\mu} \leq \hat \rho(\mathcal M)$. Since  $\rho_{p,
\mu}$ is non-decreasing with respect to $p$ by
Lemma~\ref{lemma:BasicProertiesOfPradius} we have 
\begin{equation}\label{eq:limiting leq}
\lim_{p\to\infty}\rho_{p, \mu} \leq \hat \rho (\mathcal M). 
\end{equation}
It is then natural to ask when the equality holds in this inequality. We
will show that the equality still holds under the following assumption,
which is weaker than the one in \cite{Blondel2005}. 

\begin{assumption}\label{assm:onmu}
\ 
\begin{enumerate} \renewcommand{\theenumi}{\alph{enumi}}
\item \label{item:limiting:invariant} $\mathcal M$ leaves a proper cone
invariant;

\item \label{item:decomp_nu}  The singular part $\mu_s$ of $\mu$
consists of only point measures, i.e., $\mu_s = \sum_{i=1}^N p_i
\delta_{M_i}$ for some positive numbers $p_1, \dotsc, p_N$ and matrices
$M_1, \dotsc, M_N$. 
\end{enumerate}
\end{assumption}

The next theorem is the main result of this section. 
\begin{theorem}\label{theorem:limiting}
If $\mu$ satisfies Assumption~\ref{assm:onmu} then
\begin{equation*}
\lim_{p\to\infty}\rho_{p, \mu} 
=
\hat \rho(\mathcal M). 
\end{equation*}
\afterequation
\end{theorem}

This theorem has two corollaries, both of which can be proved easily
using Proposition~\ref{prop:stabilitychar}.  The first one shows a novel
relationship between the stability of the deterministic switched system
$\Sigma_{\mathcal M}$ and the stochastic switched system $\Sigma_\mu$. 
\begin{corollary} 
Assume that $\mu$ satisfies Assumption~\ref{assm:onmu}. Then 
$\Sigma_{\mathcal M}$ is absolutely asymptotically stable if and only if
there exists $\gamma<1$ such that $\Sigma_\mu$ is $p$th mean stable with
the decay rate at most $\gamma$; i.e., the expectation
$E[\norm{x(k)}^p]^{1/p}$ is of order $O(\gamma^{k})$ for every $p$. 
\end{corollary}

The second corollary gives an expression of joint spectral radius. 
\begin{corollary} 
If $\mu$ satisfies Assumption~\ref{assm:onmu} then
\begin{equation*}
\hat \rho(\mathcal M)
=
\lim_{p\to\infty}\rho\left( E[A^{\otimes p}] \right)^{1/p}. 
\end{equation*}
\afterequation
\end{corollary}

To prove Theorem \ref{theorem:limiting} we need the next lemma. Its
proof is omitted due to limitations of space. 
\begin{lemma}\label{lemma:proplowerbound}
Assume that $\mu$ satisfies Assumption~\ref{assm:onmu}. Then, for every
$\epsilon>0$, there exists $\delta>0$ such that
\begin{equation*}
\mu(\{X : X \geq^K (1-\epsilon)M \}) \geq \delta
\end{equation*}
for every $M\in \mathcal M$. 
\end{lemma}

Let us prove Theorem~\ref{theorem:limiting}. 

\begin{proofof}{Theorem~\ref{theorem:limiting}}
By Proposition~\ref{prop:stabilitychar} and the
inequality~\eqref{eq:limiting leq} it is sufficient to show
$\lim_{p\to\infty}\rho_{p, \mu} \geq \hat \rho (\mathcal M)$. Let
$\gamma :=\hat \rho (\mathcal M)$.  Let us take any cone linear absolute
norm $\norm{\cdot}$ with respect to the invariant cone. By
Proposition~\ref{prop:AAS:characterization}, we can show that there
exist $C>0$ and $\{M_k\}_{k=1}^\infty \subset \mathcal M$ such that
$\norm{M_k \dotsi M_1} > C\gamma^k$ for infinitely many $k$. 

Take arbitrary $\gamma_1$ and $\gamma_2$ such that
$\gamma>\gamma_1>\gamma_2$. Define $\epsilon :=
(\gamma-\gamma_1)/\gamma$ and take the corresponding $\delta>0$ given by
Lemma~\ref{lemma:proplowerbound}. Observe that, by
Lemma~\ref{lemma:ProdMonotone}, if $X_i \geq^K (1-\epsilon)M_i =
(\gamma_1/\gamma)M_i$ then $\norm{X_k \dotsi X_1}  \geq 
(\gamma_1/\gamma)^k  \norm{M_k\dotsi M_1} > C \gamma_1^k$.  Therefore, 
\begin{equation*}
\begin{aligned}
E[\norm{A_k\dotsi A_1}^p]
> &
\mu^k\left(\norm{A_k\dotsi A_1} > C\gamma_1^k\right)  \left(C \gamma_1^k\right)^p
\\	
\geq &
\left(\prod_{i=1}^k \mu(\{X_i : X_i \geq^K (1-\epsilon)M_i\})\right) C^p \gamma_1^{pk}
\\
\geq&
\delta^k C^p \gamma_1^{pk}
\end{aligned}
\end{equation*}
and hence we have $E[\norm{A_k\dotsi A_1}^p]^{1/kp} > C^{1/k}
\delta^{1/p} \gamma_1$. Choose a sufficiently large $p_0$ such that 
$\delta^{1/p_0}\gamma_1 > \gamma_2$. Then  $E[\norm{A_k\dotsi
A_1}^{p_0}]^{1/kp_0} > C^{1/k} \gamma_2$ for infinitely many $k$. This
implies $\rho_{p_0,\mu} \geq \gamma_2$ and therefore
$\lim_{p\to\infty}\rho_{p,\mu}\geq \gamma_2$. This completes the proof
since $\gamma_2$ can be made arbitrary close to $\gamma = \hat
\rho(\mathcal M)$. 
\end{proofof}

The next example gives a simple illustration of
Theorem~\ref{theorem:limiting}.
\begin{example}
Let $\mu$ be the uniform distribution on $[0,\gamma]$ for some $\gamma >
0$. We can see that $\mu$ satisfies  Assumption~\ref{assm:onmu}.  Also
clearly  $\hat{\rho}(\mathcal M) = \gamma$. On the other hand we have
$\rho_{p,\mu}^p = \rho(E[A^{\otimes p}]) = \gamma^p/(p+1)$. Therefore
$\rho_{p, \mu} = \gamma(p+1)^{-p}$ and hence $\lim_{p\to\infty} \rho_{p,
\mu} = \gamma = \hat \rho(\mathcal M)$, as expected. We remark that,
since $\supp \mu$ is infinite, we cannot apply the result in
\cite{Blondel2005} to this example. 
\end{example}

Finally the next theorem generalizes another characterization of joint
spectral radius in \cite{Xu2011},  which does not need the existence of
an invariant cone. 
\begin{theorem}
If $\mu$ satisfies the condition \iref{item:decomp_nu} of
Assumption~\ref{assm:onmu}  then
\begin{equation}\label{eq:limsup}
\hat \rho (\mathcal M) 
=
\limsup_{p\to\infty} \left( \rho( E[A^{\otimes p}] )\right)^{1/p}. 
\end{equation}
\afterequation
\end{theorem}

\begin{proofsk}
Using Theorem~\ref{theorem:limiting} and the semidefinite lifting
\cite{Blondel2005} of matrices one can show $\hat \rho (\mathcal M) =
\lim_{p\to\infty} \rho(E[X^{\otimes (2p)}])^{1/(2p)}$.  Also it is
possible to see that the sequence $\{ \left( \rho( E[A^{\otimes p}]
)\right)^{1/p} \}_{p=1}^\infty$ is not decreasing. These observations
yield \eqref{eq:limsup}. 
\end{proofsk}

\section{Markovian Case}\label{sec:Markov}

So far we have restricted our attention to the case when the random
variables $\{A_k\}_{k=0}^\infty$ are identically and independently
distributed.  In this section we study a more practical case of when the
process~$\{A_k\}_{k=0}^\infty$ is a time-homogeneous Markov chain. Let
$\{\sigma_k\}_{k=0}^\infty$ be a time-homogeneous Markov chain taking
values in $\{1, \dotsc, N\}$ with the transition probability matrix
$\mathbb{P} = [p_{ij}]$ and the constant initial state $\sigma_0 \in
\{1, \dotsc, N\}$.  Let $\mathcal M = \{ M_1, \dotsc, M_N\}$ be a set of
$d\times d$ square matrices. We define the stochastic process $A =
\{A_k\}_{k=0}^\infty$ by $A_k = M_{\sigma_k}$ and the switched system
$\Sigma_A$ by 
\begin{equation*}
\Sigma_A : x(k+1) = A_k x(k). 
\end{equation*}
The $p$th mean and moment stability of $\Sigma_A$ is defined in the same
way as Definition~\ref{def:MeanStability}.  Also we define the
$L^p$-norm joint spectral radius~$\rho_{p, A}$ of~$A$ by
\eqref{eq:def:pthradius}. 

The next theorem summarizes the relationship between $p$-radius, mean
stability, and moment stability for the Markovian switched
system~$\Sigma_A$. Though the result actually holds for a general $p$,
for simplicity of presentation we will restrict our attention to $p =1$
or $2$. 

\begin{theorem}\label{theorem:MarkovianStability}
Assume that either $p=2$, or $\mathcal M$ leaves a proper cone $K$
invariant and $p=1$. Then the following conditions are equivalent. 
\begin{enumerate} \renewcommand{\theenumi}{\roman{enumi}}
\item \label{item:rho<1:Markov}
$\rho_{p, A} < 1$ for every $\sigma_0$;

\item \label{item:meanst:Markov}
$\Sigma_A$ is $p$th mean stable;

\item \label{item:momentst:Markov}
$\Sigma_A$ is $p$th moment stable;

\item \label{item:spec<1:Markov}
Define the $Nd^p \times Nd^p$ matrix $T_p$ by 
\begin{equation*}
T_p = (\mathbb{P}^\top \otimes I_{d^p}) 
\diag(M_1^{\otimes p}, \dotsc, M_N^{\otimes p}). 
\end{equation*}
Then $\rho(T_p) < 1$. 
\end{enumerate}
\end{theorem}

To prove this theorem we introduce the following
definitions~\cite{Costa2004}. Let $x(\cdot;x_0,\sigma_0)$ denote the
trajectory of $\Sigma_A$ with the initial conditions $x_0$ and
$\sigma_0$. For each $k$ define $Q(k) = (Q_i(k), \dotsc, Q_N(k)) \in
(\mathbb{R}^d)^N$ by $Q_i(k;x_0,\sigma_0) :=
E[x(k;x_0,\sigma_0)_{\sigma_k=i}]$, and the operator $\mathcal L$ on
$(\mathbb{R}^d)^N$ by  $\mathcal L(H) := (\mathcal L_1(H), \dotsc,
\mathcal L_N(H))$ and  $\mathcal L_j(H) := \sum^n_{i=1} p_{ij} A_i H_i$.
The proof of the next lemma is omitted because it can be shown in the
same way as \cite{Costa2004}. 
\begin{lemma}\label{lemma:LQvecConnection}
Let $\mathcal L$ and $Q$ be as above. 
\begin{enumerate}
\item $E[x(k;x_0,\sigma_0)] = \sum_{j=1}^N Q_j(k;x_0,\sigma_0)$. 

\item $Q(k+1;x_0,\sigma_0) = \mathcal L(Q(k;x_0,\sigma_0))$. 

\item For every $H\in (\mathbb{R}^d)^N$, $\vec{\mathcal L(H)}
= T_1(\vec H)$. 
\end{enumerate}
\end{lemma}

With this lemma we can prove Theorem~\ref{theorem:MarkovianStability}. 

\begin{proofof}{Theorem~\ref{theorem:MarkovianStability}}
The equivalence \iref{item:rho<1:Markov} $\Leftrightarrow$
\iref{item:meanst:Markov} $\Leftrightarrow$ \iref{item:momentst:Markov}
can be proved in the same way as Proposition \ref{prop:stabilitychar}. 
Also, when $p=2$, the equivalence \iref{item:meanst:Markov}
$\Leftrightarrow$ \iref{item:spec<1:Markov} is proved in
\cite{Costa2004}. Therefore it is sufficient to show the implications
\iref{item:spec<1:Markov} $\Rightarrow$ \iref{item:momentst:Markov} and
\iref{item:meanst:Markov} $\Rightarrow$ \iref{item:spec<1:Markov} under
the assumption that $p=1$ and $\mathcal M$ leaves a proper cone $K$
invariant

[\iref{item:spec<1:Markov} $\Rightarrow$
\iref{item:momentst:Markov}]:\quad  Assume $\rho(T_1) < 1$ and let $x_0$
and $\sigma_0$ be arbitrary. Then, by Lemma \ref{lemma:LQvecConnection},
\begin{equation*}
\begin{aligned}
\vec Q(k;x_0,\sigma_0) 
&= \vec(\mathcal{L}^k(Q(0;x_0,\sigma_0))) 
\\
&= T_1^k \vec Q(0;x_0,\sigma_0) \to 0, 
\end{aligned}
\end{equation*}
which shows $Q(k;x_0,\sigma_0) \to 0$ as $k\to\infty$ and therefore
$E[x(k;x_0,\sigma_0)]\to 0$  for all $x_0$ and $\sigma_0$. Thus
$\Sigma_A$ is first moment stable. 

[\iref{item:meanst:Markov} $\Rightarrow$
\iref{item:spec<1:Markov}]:\quad Omitted due to limitations of space. 
\end{proofof}

The next corollary of Theorem~\ref{theorem:MarkovianStability}  enables
us to compute $p$-radius efficiently. 
\begin{corollary} \label{corollary:prad:computation:Markov}
If $p$ and $\mu$ satisfies the assumption in
Theorem~\ref{theorem:MarkovianStability} then it holds that 
\begin{equation*}
\rho_{p, A} = \rho(T_p)^{1/p}. 
\end{equation*}
\end{corollary}

Finally let us apply the results obtained in this section to the
stabilization of Markovian switched systems 
Let 
\begin{equation*}
\begin{gathered}
N=3,\ \mathbb{P} = 
\begin{bmatrix}
0.3 &  0.5 &  0.2\\
0.5 &  0.3 &  0.2\\
0.2 &  0.2 &  0.6
\end{bmatrix},\\
M_1 = \begin{bmatrix}
.32 & .49\\ .24&.33
\end{bmatrix},\ 
M_2 = \begin{bmatrix}
.53 & .65\\ .75& .85
\end{bmatrix},\ 
M_3 = \begin{bmatrix}
1.5 & .51\\ .18& .69
\end{bmatrix}. 
\end{gathered}
\end{equation*}

Corollary~\ref{corollary:prad:computation:Markov} gives $\rho_{1,A} =
1.221$ and therefore $\Sigma_A$ is not first mean stable. Let us
consider the stabilization of $\Sigma_A$. Define the switched system
with input by $x(k+1) = A_{k}x(k) + b_{k}u(k)$ where $b_k =
n_{\sigma_k}$ with
\begin{equation*}
n_1 = \begin{bmatrix}
-0.56 \\ 0.39
\end{bmatrix},\ 
n_2 = \begin{bmatrix}
0.40 \\ -1.70
\end{bmatrix},\ 
n_3 = \begin{bmatrix}
-0.37\\ -0.49
\end{bmatrix}.
\end{equation*}
As an input we use the static state feedback $u(k) = fx(k)$ for some $f
\in \mathbb{R}^{1\times 2}$. This yields the controlled system
\begin{equation*} 
\Sigma_{A+bf}: x(k+1) = (A_{k} + b_{k}f)x(k) 
\end{equation*}

Let $f = \bm{0.36&0.50}$. Then all the matrices $M_i+n_if$ ($i=1, 2, 3$)
have only positive entries. Therefore we can use
Corollary~\ref{corollary:prad:computation:Markov} to find $\rho_{1,
A+bf} = 0.9554$. Therefore the controlled system $\Sigma_{A+bf}$ is
first mean stable by Theorem~\ref{theorem:MarkovianStability}.
Fig.~\ref{fig:figure}
\begin{figure}[tb]
\begin{center}
\includegraphics[scale=\myscale]{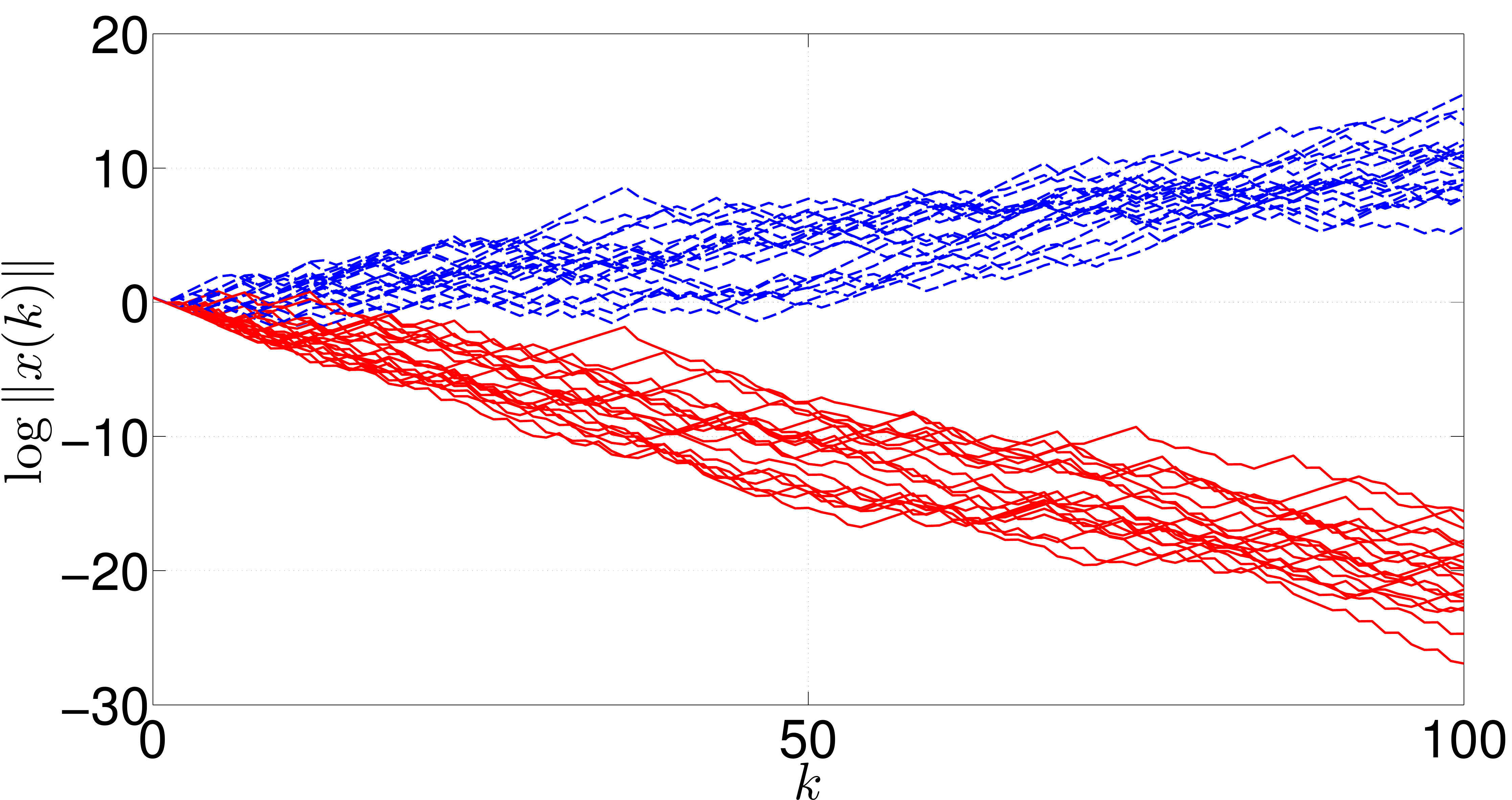} 
\caption{{Sample paths of switched systems. Dashed: Before stabilization.
Solid:  After stabilization}}
\label{fig:figure}
\end{center}
\end{figure}
shows the 20 sample paths of the original switched system and the
stabilized switched system. Finding a systematic way to obtain a
stabilizing feedback gain is left as an open problem.

\section{Conclusion}

We investigated the mean stability of a class of  discrete-time
stochastic switched systems. First we presented the equivalence between
mean stability and the existence of a homogeneous Lyapunov function.
Then we showed that, in the limit of $p\to\infty$, the $p$th mean
stability becomes equivalent to the absolute asymptotic stability of an
associated  deterministic switched system.  Finally the characterization
of the stability of a class of Markovian switched systems was given. 
Throughout the paper $L^p$-norm joint spectral radius has played a key
role.


\end{document}